\documentclass[12pt,letterpaper,reqno]{amsart}

\usepackage{amsmath}
\usepackage{amssymb}
\usepackage{amsthm}
\usepackage{indentfirst}
\usepackage{xspace}
\usepackage{multirow}
\usepackage{hyperref}
\usepackage{xcolor}
\definecolor{darkblue}{rgb}{0,0,0.4}
\hypersetup{pdfauthor={David Dumas and Andrew Neitzke},pdftitle={Asymptotics of Hitchin's metric on the Hitchin section},colorlinks=true,urlcolor=darkblue,linkcolor=darkblue,citecolor=darkblue}
\usepackage{verbatim}
\usepackage[letterpaper,margin=1in]{geometry}
\usepackage{thmtools}
\usepackage[all]{xy}
\usepackage{longtable}
\usepackage{capt-of}
\usepackage{enumitem}
\usepackage{graphicx}

\newlength{\bibitemsep}
\newlength{\bibparskip}
\let\oldthebibliography\thebibliography
\renewcommand\thebibliography[1]{%
  \oldthebibliography{#1}%
  \setlength{\parskip}{\bibitemsep}%
  \setlength{\itemsep}{\bibparskip}%
}

\setlist{itemsep = 0.25em, topsep = 0.25em}

\declaretheoremstyle[spaceabove=0.25cm,spacebelow=0.25cm,notefont=\normalfont\bfseries, notebraces={(}{)}]{theorem}
\declaretheoremstyle[spaceabove=0.25cm,spacebelow=0.25cm,bodyfont=\normalfont,notefont=\normalfont\bfseries, notebraces={(}{)}]{noital}
\declaretheoremstyle[spaceabove=0.25cm,spacebelow=0.25cm,bodyfont=\normalfont\color{darkgreen},notefont=\normalfont\bfseries, notebraces={(}{)}]{green}
\declaretheoremstyle[spaceabove=0.25cm,spacebelow=0.25cm,bodyfont=\normalfont,notefont=\normalfont\bfseries,qed=$\qedsymbol$,notebraces={(}{)}]{proofstyle}

\declaretheorem[name=Theorem,style=theorem]{thm}

\declaretheorem[name=Lemma,style=theorem,sibling=thm]{lem}

\numberwithin{equation}{section}

\newcommand{\cB}{\ensuremath{\mathcal B}}

\newcommand{\cL}{\ensuremath{\mathcal L}}

\newcommand{\cM}{\ensuremath{\mathcal M}}

\newcommand{\cE}{{\mathcal E}}

\newcommand{\R}{\ensuremath{\mathbb R}}
\newcommand{\C}{\ensuremath{\mathbb C}}
\newcommand{\PP}{\ensuremath{\mathbb P}}

\newcommand{\tensor}{\otimes}

\newcommand{\noproof}{\hfill\qedsymbol}

\newcommand{\half}{\ensuremath{\frac{1}{2}}}

\newcommand{\kahler}{K\"ahler\xspace}
\newcommand{\hk}{hyperk\"ahler\xspace}

\newcommand{\I}{{\mathrm i}}
\newcommand{\e}{{\mathrm e}}
\newcommand{\de}{\mathrm{d}}
\renewcommand{\sf}{{\mathrm {sf}}}
\newcommand{\near}{{\mathrm {near}}}
\newcommand{\far}{{\mathrm {far}}}

\newcommand{\SL}{{\mathrm {SL}}}

\newcommand{\fsu}{{\mathfrak {su}}}

\newcommand{\abs}[1]{\lvert#1\rvert}
\newcommand{\norm}[1]{\lVert#1\rVert}

\newcommand{\eps}{\epsilon}

\newcommand{\ti}[1]{\textit{#1}}

\renewcommand{\leq}{\leqslant}
\renewcommand{\geq}{\geqslant}
\renewcommand{\ge}{\geqslant}
\newcommand{\bdry}{\partial}
\newcommand{\ident}{\equiv}

\DeclareMathOperator{\im}{Im}
\DeclareMathOperator{\re}{Re}
\DeclareMathOperator{\tr}{tr}
\DeclareMathOperator{\Tr}{Tr}

\DeclareMathOperator{\End}{End}

\DeclareMathOperator{\SU}{SU}

\newcommand{\insfig}[3]{\begin{figure}[htbp] \centering \includegraphics[scale=#2]{#1.pdf} \caption{#3} \label{fig:#1} \end{figure}}

\newenvironment{rmenumerate}{\begin{enumerate}}{\end{enumerate}}

\begin{document}

\setcounter{page}{1}

\title{Asymptotics of Hitchin's metric on the Hitchin section}
\author[D. Dumas]{David Dumas}
\author[A. Neitzke]{Andrew Neitzke}
\date{March 27, 2018.  (v1: February 20, 2018)}

{\abstract{We consider Hitchin's \hk metric $g$ on the moduli space
    $\cM$ of degree zero $\SL(2)$-Higgs bundles over a compact Riemann surface.
    It has been conjectured that, when
    one goes to infinity along a generic ray in $\cM$, $g$ converges to
    an explicit ``semiflat'' metric $g^\sf$, with an exponential rate of
    convergence. We show that this is indeed the case for the
    restriction of $g$ to the tangent bundle of the 
    Hitchin section $\cB \subset \cM$.}}

\maketitle

\section{Introduction}

\subsection{Summary}
Fix a compact Riemann surface $C$. In \cite{hitchin87a} Hitchin studied
the moduli space $\cM$ of degree zero \ti{$\SL(2)$-Higgs bundles on $C$},
and showed in particular that $\cM$ admits a canonically defined \hk metric
$g$.

In \cite{Gaiotto:2008cd,Gaiotto:2009hg} a new conjectural construction of $g$ was given.
The full conjecture is complicated to state (see \cite{notes-hk} for a review), 
but one of its consequences is a concrete picture of the generic asymptotics
of $g$, as follows.

The non-compact space $\cM$ is fibered over the space $\cB$ of holomorphic 
quadratic differentials on $C$. We consider
a path to infinity in $\cM$, lying over a generic ray
$\{t \phi_0\}_{t \in \R_+} \subset \cB$, where $\phi_0$ has only
simple zeroes. Along such a path, the prediction is that
\begin{equation} \label{eq:intro-prediction}
  g = g^\sf + O\left(\e^{-4 \alpha t^\half}\right),
\end{equation}
where $g^\sf$ is the \ti{semiflat} metric, given by a simple
explicit formula (see \S\ref{sec:semiflat}), and $\alpha$
is any constant with
$\alpha < M(\phi_0)$, where $M(\phi_0)$ is the length
of the shortest saddle connection in the metric $\abs{\phi_0}$
(see \S\ref{sec:threshold}).

Very recently Mazzeo-Swoboda-Weiss-Witt \cite{Mazzeo2017} have shown
that, along a generic ray, the difference $g - g^\sf$ does
decay at least \ti{polynomially} in $t$.
This work motivated us to wonder whether
one could show directly that the decay is actually exponential.
In this paper we show that this is indeed the case for
the restriction of $g$ to the tangent bundle of a 
certain embedded copy of $\cB$ inside $\cM$, the \ti{Hitchin section}:
\eqref{eq:intro-prediction} holds there 
for any $\alpha < \half M (\phi_0)$. (Unfortunately, we miss the conjectured 
sharp constant by a factor of $2$.)
The precise statement is given in \autoref{thm:main} below.

\subsection{The strategy} \label{sec:strategy}

Points of $\cB$ correspond to holomorphic quadratic differentials $\phi$ on $C$.
Since these form a linear space, tangent vectors to $\cB$ likewise correspond
to holomorphic quadratic differentials $\dot\phi$.
Given $(\phi, \dot\phi) \in T\cB$,
both $g_\phi(\dot\phi, \dot\phi)$ and $g_\phi^\sf(\dot\phi, \dot\phi)$ 
arise as integrals over $C$ (which can be found in \eqref{eq:l2norm-combined}
and \eqref{eq:l2sfnorm-combined} below).
The integrand in $g_\phi^\sf$ is completely explicit, while the integrand in $g_\phi$ 
depends on the solutions of two elliptic scalar PDEs on the surface $C$.
To prove \eqref{eq:intro-prediction} for some given $\alpha$, we need to
show that these two integrals agree up to $O(\e^{-4 \alpha t^\half})$.

\insfig{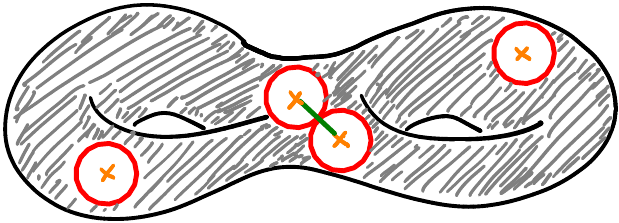}{0.85}{A genus $2$ surface $C$ equipped with a holomorphic 
quadratic differential $\phi_0$ which has $4$ simple zeroes (orange crosses). 
The shortest saddle connection is shown in green; its length is $M(\phi_0)$.
We have chosen $\alpha$ slightly smaller than $\half M(\phi_0)$. 
$C_\near$ is the union of $4$ disks $D_i$ centered on the zeroes, all with the 
same radius $\alpha$. The complementary region $C_\far$ is shaded.}

To do this, we let $r_0(z)$ denote the $\abs{\phi_0}$-distance from $z$ to the closest zero of $\phi_0$, and
divide the surface $C$ into two regions, as illustrated
in \autoref{fig:surface-crop}:
\begin{itemize}
\item The ``far'' region $C_{\far} = \{z: r_0(z) > \alpha\}$.
In this region we can show that the \ti{integrands} agree
to order $O(\e^{-4 \alpha t^\half})$: indeed,
we show that the difference $\delta$ of the integrands
decays as $\delta = O(\e^{- \gamma t^\half r_0(z)})$ for any
$\gamma < 4$. This part of our analysis contains no big surprises,
and is closely parallel
to the analysis carried out by Mazzeo-Swoboda-Weiss-Witt
in the more general setup of arbitrary $\SL(2)$-Higgs bundles in
\cite{Mazzeo2017}.
(However, because we restrict to the Hitchin section $\cB \subset \cM$, 
our job is somewhat simpler: we only have to deal with scalar PDEs,
and use more-or-less standard techniques.
The specific estimates we use in this part are built on the work of
Minsky in \cite{Minsky92}.)

\item The ``near'' region
$C_{\near} = \{z: r_0(z) \le \alpha\}$.  This region
looks more difficult because our estimates do not show that $\delta$
is close to zero here.  The happy surprise---which was really the
reason for writing this paper---is that when $\alpha < \half M(\phi_0)$,
$\delta$ turns out to be close to an \ti{exact} form that we can control, as follows. 
For any $\alpha < \half M(\phi_0)$, $C_{\near}$ is a disjoint union of disks
$D_i$ centered on the zeros of $\phi_0$.  On each $D_i$ we show
that $\delta = \de \beta_i + O(\e^{-4 \alpha t^\half})$, for a
$1$-form $\beta_i$ which has the same decay property as $\delta$,
namely $\beta_i = O(\e^{- \gamma t^\half r_0(z)})$.  Thus $\beta_i$ is
exponentially small on the boundary of $D_i$, and Stokes's theorem
gives
$\int_{D_i} \delta = \int_{\partial D_i} \beta_i + O(\e^{-4 \alpha
  t^\half}) = O(\e^{- 4 \alpha t^\half})$.
\end{itemize}

Combining these contributions we obtain the desired estimate $\int_{C} \delta = O(\e^{-4 \alpha t^\half})$.

\subsection{Outline}

We carry out the strategy described above as follows.  In
\S\S\ref{sec:background}-\ref{sec:estimate} we set up the background
and notation, and state our main result precisely, as
\autoref{thm:main}. In \S\ref{sec:coordinate-computations} we derive
integral expressions for the restrictions of $g$ and $g^\sf$
to $T \cB$. In \S\S\ref{sec:decay}-\ref{sec:F-estimates} we
develop the main PDE estimates we use to derive exponential decay. In
\S\S\ref{sec:holomorphic}-\ref{sec:exactness} we construct the
$1$-forms $\beta_i$ which we use
in the ``near'' region. In \S\ref{sec:final} we put all
this together to complete the proof of the main theorem.

\subsection{Origin in experiment}

This work was initially inspired by computer experiments (using
programs developed by the authors, and building on work of the first
author and Wolf in \cite{blaschke}) that seemed to show exponential
decay of $g-g^{\sf}$ in certain cases, despite the lack of an
exponentially decaying bound on the integrand near the zeros of
$\phi$.  While these experiments were conducted in a slightly different
setting---namely, meromorphic Higgs bundles on $\C\PP^1$ with a single
pole---all of the essential features and challenges are present in
both cases.  The experimental results therefore suggested that some
``cancellation'' would occur in $C_{\near}$.  Further investigation
of the integrand in this region led to the results of
\S\S\ref{sec:holomorphic}-\ref{sec:exactness} below, and thus to the
main theorem.

This experimental counterpart of this work is ongoing and will be the
subject of a forthcoming paper and software release.

\subsection{Outlook}

It would be very desirable to understand how to extend
\autoref{thm:main} to Higgs bundles of higher rank, say $\SL(N)$-Higgs
bundles.  There is a conjecture very similar to
\eqref{eq:intro-prediction} in that case, but instead of the shortest
saddle connection, it involves the \ti{lightest finite web} as defined
in \cite{Gaiotto2012}. While the analysis of $C_{\far}$ should extend
to this case using methods similar to those of \cite{Mazzeo2017}, 
it is not clear how our approach to $C_{\near}$ should be generalized.

Similarly, one would like to extend \autoref{thm:main} to work on the
full $\cM$ instead of only $\cB \subset \cM$. The analysis of
$C_{\far}$ has already been done on the full $\cM$ in
\cite{Mazzeo2017}, so again the issue is whether the analysis of
$C_{\near}$ can be extended.

In another direction, it would be desirable to improve \autoref{thm:main} to show
that the exponential estimate holds for all $\alpha < M(\phi_0)$
instead of just $\alpha < \half M(\phi_0)$.  However, this might
require a new method; in our computation we meet several different
corrections which are naively of the same order
$\e^{-2 M(\phi_0) t^\half}$; one would need to find some mechanism by
which these different corrections can cancel one another.

\subsection{Acknowledgements}

The authors thank Rafe Mazzeo, Jan Swoboda, Hartmut Weiss, and
Michael Wolf for helpful discussions related to this work, and
also thank the anonymous referee for a careful reading and
helpful comments and corrections. The
authors also gratefully acknowledge support from the U.S.~National
Science Foundation through individual grants DMS 1709877 (DD), DMS
1711692 (AN), and through the GEAR Network (DMS 1107452, 1107263,
1107367, ``RNMS: GEometric structures And Representation varieties'')
which supported a conference where some of this work was conducted.

\section{Background} \label{sec:background}

\subsection{Higgs bundles}

Recall that a \ti{stable $\SL(2)$-Higgs bundle over $C$ of degree zero}
is a pair $(\cE, \varphi)$ where
\begin{itemize}
\item $\cE$ is a rank $2$ holomorphic vector bundle over $C$, equipped
with a trivialization of $\det \cE$,
\item $\varphi$ is a traceless holomorphic section of $\End \cE \otimes K_C$,
\item all $\varphi$-invariant subbundles of $\cE$ have negative degree.
\end{itemize}
There is a (coarse) moduli space $\cM$ parameterizing stable $\SL(2)$-Higgs 
bundles over $C$ of degree
zero \cite{hitchin87a,hitchin87b}.

\subsection{Harmonic metrics}

For each stable $\SL(2)$-Higgs bundle $(\cE, \varphi)$ of degree zero,
it is shown in \cite{hitchin87a} that
there is a distinguished unit-determinant Hermitian metric $h$ on
$\cE$, the \ti{harmonic metric}.  The metric $h$ is determined by
solving an elliptic PDE: letting $D$ denote the Chern connection in
$\cE$, with curvature $F_{D} \in \Omega^2(\fsu(\cE,h))$, and letting
$\Phi = \varphi - \varphi^\dagger \in \Omega^1(\fsu(\cE,h))$ with
$\varphi^\dagger$ the $h$-adjoint of $\varphi$, we require
\begin{equation}
\label{eq:sde}
  F_{D} - \half [\Phi, \Phi] = 0.
\end{equation}
In this equation both $F_D$ and $\Phi$ depend on $h$.

\subsection{The \hk metric}

Now we recall Hitchin's \hk metric $g$ on the moduli space $\cM$.
A beautiful description of this metric was given by Hitchin in 
\cite{hitchin87a} in terms
of an infinite-dimensional \hk quotient. In this paper we
will not use the \hk structure; all we need is a practical
recipe for computing the metric. In this section we review that recipe.

Let $v$ be tangent to an arc in $\cM$, and lift this arc to a family
of Higgs bundles $(\cE_t, \varphi_t)$, equipped with harmonic metrics
$h_t$. Identify all the $(\cE_t, h_t)$ with a fixed $C^\infty$
$\SU(2)$-bundle $E$.  Then we have a family of unitary connections
$D_t$ on $E$ and 1-forms $\Phi_t \in \Omega^1(\fsu(E))$ which for all
$t$ satisfy \eqref{eq:sde}.  For brevity, let $D := D_0$ and
$\Phi := \Phi_0$ denote these objects at $t=0$.  Differentiating at $t=0$
we obtain a pair of $1$-forms
\begin{equation}
(\dot A, \dot \Phi) =   (\left . \partial_t D_t \right
|_{t=0}, \left . \partial_t \Phi_t \right |_{t=0}) \in \Omega^1(\fsu(E))^2.
\end{equation}

Given $\alpha \in \Omega^1(\fsu(E))$ we define a nonnegative density
$\abs{\alpha}^2$ on $C$ by
\begin{equation}
\abs{\alpha_x \de x + \alpha_y \de y}^2 = - \Tr(\alpha_x^2 + \alpha_y^2) \, \de x \de y.
\end{equation}
Here $z = x + \I y$ is a local conformal coordinate on $C$.
In coordinate-independent terms, the density $|\alpha|^2$ corresponds (using the
orientation of $C$) to the $2$-form $-\Tr(\alpha \wedge \star
\alpha)$, where $\star$ denotes the Hodge star operator on $1$-forms.
Now we equip $\Omega^1(\fsu(E))^2$ with the $L^2$ metric
\begin{equation}
  \norm{(\dot A, \dot \Phi)}^2 = \int_C \left (\abs{\dot A}^2  +
  \abs{\dot \Phi}^2 \right ).
\end{equation}
Let $\rho: \Omega^0(\fsu(E))
\to \Omega^1(\fsu(E))^2$ be the linearized gauge map, defined by
\begin{equation}
\label{eq:def-rho}
\rho(X) = (-\de_D X, [X, \Phi] ).
\end{equation}
We consider the orthogonal decomposition of $(\dot
A, \dot \Phi)$ relative to the image of $\rho$,
\begin{equation}
\label{eq:orthog-decomp}
 (\dot A, \dot \Phi) = (\dot A, \dot \Phi)^\parallel + (\dot A, \dot
\Phi)^\perp
\end{equation}
with $(\dot A, \dot \Phi)^\parallel \in \rho(\Omega^0(\fsu(E))$ and $
(\dot A, \dot \Phi)^\perp \in \rho(\Omega^0(\fsu(E))^\perp$.
Hitchin's \hk metric $g$ is
\begin{equation}
 g(v,v) = \norm{(\dot A, \dot \Phi)^\perp}^2.
\end{equation}

\subsection{The Hitchin section}

Fix a spin structure on the compact Riemann surface $C$.
The spin structure determines a holomorphic line bundle $\cL$ equipped
with an isomorphism $\cL^2 \simeq K_C$, and thus a rank $2$ holomorphic
vector bundle
\begin{equation}
  \cE = \cL \oplus \cL^{-1}.
\end{equation}
This bundle has $\det \cE = \cL \otimes \cL^{-1}$ which is canonically trivial.
Let $\cB$ be the space of holomorphic quadratic differentials on $C$,
\begin{equation}
 \cB = H^0(C, K_C^2).
\end{equation}
For each $\phi \in \cB$ there is a corresponding Higgs field,
\begin{equation}
  \varphi = \begin{pmatrix} 0 & - \phi \\ 1 & 0 \end{pmatrix} \in H^0(C, \End \cE \otimes K_C).
\end{equation}
The Higgs bundles $(\cE, \varphi)$ are all stable, and thus determine
a map $\iota: \cB \to \cM$. The image $\iota(\cB) \subset \cM$ is an
embedded submanifold, the \ti{Hitchin section}.\footnote{More precisely there are $2^{2 \times {\mathrm {genus}}(C)}$ Hitchin sections, corresponding
to the equivalence classes of spin structures on $C$. All of our discussion
applies to any of them.} Moreover, $\iota$ is a holomorphic map,
with respect to the complex structure on $\cM$ induced from its
realization as moduli space of Higgs bundles (which is the complex structure
denoted $I$ in \cite{hitchin87a}). Thus $\iota(\cB)$ is a complex
submanifold of $\cM$. From now on, by abuse of notation, 
we identify $\cB$ with $\iota(\cB)$.

Our interest in this paper is in the restriction of the \hk metric 
$g$ from the full $T \cM$ to $T \cB$.
This restriction is a \kahler metric on $\cB$, which we will
also denote $g$.

\section{Metric estimate} \label{sec:estimate}

\subsection{The semiflat metric} \label{sec:semiflat}

Let $\cB' \subset \cB$ be the locus of quadratic differentials with
only simple zeros, which is an open and dense set.  On $\cB'$ we
define an explicit \kahler metric $g^\sf$ as follows.  A tangent
vector to $\cB$ can be represented by a quadratic differential
$\dot \phi$. We define
\begin{equation} \label{eq:gsf}
  g^\sf_{\phi}(\dot\phi, \dot\phi) = 2 \int_C \frac{\,\,\abs{\dot\phi}^2}{\abs{\phi}}.
\end{equation}
Note that the integrand on the right hand side is a smooth density on
$C \setminus \phi^{-1}(0)$.  The condition that $\phi \in \cB'$
implies that this integral is convergent.

We remark that $g^\sf$ is a ``(rigid) special \kahler'' metric on
$\cB'$ in the sense of \cite{Freed:1997dp}. It does not extend to a Riemannian
metric on the full $\cB$.

\subsection{Threshold and radius} \label{sec:threshold}
Any nonzero quadratic differential $\phi \in \cB$ induces a
flat metric $\abs{\phi}$ on $C$, which is smooth except for conical
singularities at the zeros of $\phi$.  From now on we always use this
metric to define geodesics and lengths on $C$, unless a different
metric is explicitly referenced.  A \ti{saddle
connection} of $\phi$ 
is a geodesic segment on $C$ which begins and ends on zeros of
$\phi$ (not necessarily two distinct zeros), and which has no zeros
of $\phi$ in its interior. 

We define the \ti{threshold} $M: \cB \to \R_{\ge 0}$ by
\begin{equation}
	M(\phi) = \begin{cases} \text{the minimum length of a saddle connection of $\phi$} & \text { for } \phi \in \cB', \\
	0 & \text{ for } \phi \in \cB \setminus \cB'. \end{cases}
\end{equation}
Then $M$ is continuous and has the homogeneity property
\begin{equation}
M(t \phi) = t^{\half} M(\phi), \qquad t \in \R_+.
\end{equation}

The threshold measures the distance ``between zeros'' of $\phi$
(including the possibility of a segment between a zero and itself).
In what follows it will also be important to consider the distance
from an arbitrary point to the zeros of $\phi$.  We define the
\emph{radius} function $r : C \to \R$ of $\phi$ by
\begin{equation} r(z) = d(z,\phi^{-1}(0)). \end{equation}
The main technical estimates that are used in the proof of
\autoref{thm:main} are all phrased in terms of bounds on various functions
on $C$ in terms of the radius.

\subsection{The estimate}
Now we can state the main result of this paper:
\begin{thm}
\label{thm:main}
If $\phi_0 \in \cB'$, and $\dot\phi$ is any holomorphic quadratic
differential on $C$, then for any $\alpha < \half M(\phi_0)$ we have
\begin{equation}
\label{eqn:main}
\abs{(g_{t\phi_0} - g_{t\phi_0}^\sf)(\dot\phi,\dot\phi)}
 = O(\e^{- 4 \alpha \,t^{\half}} \|\dot\phi\|^2 )
\end{equation}
 as $t \to \infty$, where $\| \cdot \|$ denotes any norm on the vector
 space $\cB$.  Having fixed such a norm, the implicit multiplicative
 constant in \eqref{eqn:main} can be taken to depend only on
 $\alpha$, $M(\phi_0)$, and the genus of $C$.
\end{thm}

\section{Coordinate computations} \label{sec:coordinate-computations}

\subsection{Self-duality equation and variation in coordinates}
To set the stage for the proof of \autoref{thm:main} we start by
deriving local coordinate expressions for the self-duality equation
\eqref{eq:sde} at a point $\phi \in \cB$, and for its first
variation in the direction of $(\dot A, \dot \Phi)$ representing
$\dot{\phi} \in T_\phi \cB$.

In a local conformal coordinate $z = x + \I y$ on $C$ we write
$\phi = P(z) \, \de z^2$ for a holomorphic function $P$.  Let
$\de z^{\half}$ denote a local section of $\cL$ satisfying
$\de z^{\half} \tensor \de z^\half = \de z$; there are two such local
sections, the choice of which will not matter in the sequel.  Using
the local trivialization of $\mathcal{E} = \cL \oplus \cL^{-1}$ given
by the frame $(\de z^{\half}, \de z^{-\half})$, which we call the
\ti{holomorphic gauge}, we can write
\begin{equation}
\label{eq:holo-gauge}
  \varphi = \begin{pmatrix} 0 & -P \\ 1 & 0 \end{pmatrix} \de z,
  \qquad h = \begin{pmatrix} \e^{-u} & 0 \\ 0 & \e^u \end{pmatrix},
 \qquad \varphi^{\dagger} = \begin{pmatrix} 0 & \e^{2u}  \\ -\e^{-2u} \overline{P} & 0 \end{pmatrix} \de \overline{z},
\end{equation}
\begin{equation} \label{eq:holo-gauge-2}
D = \de + A, \qquad A = \begin{pmatrix} -\partial u & 0 \\ 0 & \partial u \end{pmatrix}.
\end{equation}
This diagonal form for $h$ reflects that the splitting $\cL
\oplus \cL^{-1}$ is orthogonal for the harmonic metric in this case
\cite[Theorem~11.2]{hitchin87a}.

Then \eqref{eq:sde} reduces to a scalar equation for $u$,
\begin{equation} \label{eq:diff-u}
\Delta u - 4(\e^{2u}  - \e^{-2u}\abs{P}^2 ) = 0,
\end{equation}
where $\Delta = 4 \partial \bar{\partial}$ is the flat Laplacian.

In more invariant terms, \eqref{eq:diff-u} is an equation for the
globally defined metric $\e^{2u} |\de z|^2$ on $C$.  For Higgs bundles
of this type, the Hermitian metric $h$, the K\"ahler metric
$\e^{2u} |\de z|^2$, and the (local) scalar function $u$ all contain
equivalent information.  In most of what follows we work with $u$, which
unlike $h$ and $\e^{2u} |\de z|^2$ is a coordinate-dependent quantity:
Under a conformal change of coordinates $z \mapsto w$ it transforms as
$u \mapsto u - \log \left | \frac{\de w}{\de z} \right |$.  We refer to
objects with this transformation property as \emph{log densities}.  Note
that the difference of two log densities is a function.  Also, if
$\phi = P \,\de z^2$ is a quadratic differential, then $\half \log |P|$ is
a log density.

When considering the density $u$ on $C$ which corresponds to the unique
harmonic metric on the Higgs bundle associated to $\phi \in \cB$, we
sometimes write $u(\phi)$ to emphasize its dependence on $\phi$, and to
distinguish it from other local solutions to \eqref{eq:diff-u} on
domains in $C$ or in the plane that we consider.

Next, we consider a variation $\dot{\phi} \in T_\phi \cB$ expressed
locally as $\dot \phi = \dot P(z) \, \de z^2$.
Differentiating \eqref{eq:diff-u} we find that the corresponding
first order variation $\dot u$, describing the infinitesimal change in $h$,
satisfies the inhomogeneous linear equation
\begin{equation} \label{eq:diff-udot}
 \Delta \dot u - 8 \dot u (\e^{2u} + \e^{-2u} \abs{P}^2) + 8 \e^{-2u} \re(P \dot {\overline {P}}) = 0.
\end{equation}
Unlike $u$, $\dot u$ is a well-defined global function on $C$
(independent of the coordinate $z$).
Since the operator $\Delta - 8(\e^{2u} + \e^{-2u}\abs{P}^2)$ is
negative definite, \eqref{eq:diff-udot} uniquely determines $\dot u$.

\subsection{Unitary gauge}
In preparation for calculating the $L^2$ inner product of variations
it is more convenient to work in \ti{unitary gauge}, expressing the
Higgs field and connection relative to the frame $(\e^{\half u}
\de z^{\half}, \e^{- \half u} \de z^{-\half})$; then
\eqref{eq:holo-gauge}-\eqref{eq:holo-gauge-2} become
\begin{equation}
\label{eq:unitary-gauge}
\varphi = \begin{pmatrix} 0 & -\e^{-u} P \\ \e^u & 0 \end{pmatrix} \de z, \qquad \varphi^\dagger = \begin{pmatrix} 0 & \e^u \\ -\e^{-u} \overline{P} & 0 \end{pmatrix} \de \overline{z}, \quad A = \frac{\I}{2}  \begin{pmatrix} \star \de u & 0 \\ 0 & - \star \de u \end{pmatrix},
\end{equation}
with infinitesimal variations given by
\begin{subequations} \label{eq:unitary-gauge-dot}
\begin{gather}
\dot \varphi = \begin{pmatrix} 0 & \e^{-u} P \dot{u} - \e^{-u} \dot{P} \\ \e^u \dot{u} & 0 \end{pmatrix} \de z, \qquad \dot \varphi^\dagger = \begin{pmatrix} 0 & \e^u \dot u \\ \e^{-u} \overline{P} \dot u - \e^{-u} \dot{\overline{P}} & 0 \end{pmatrix} \de \overline{z}, \\ \dot A = \frac{\I}{2} \begin{pmatrix} \star \de \dot u & 0 \\ 0 & - \star \de \dot u \end{pmatrix},
\end{gather}
\end{subequations}
which of course gives a corresponding expression for $\dot{\Phi} =
\dot \varphi - \dot \varphi^\dagger$.

\subsection{Orthogonal decomposition}
Let $(A,\Phi)$ be obtained from a solution of the self-duality equation \eqref{eq:sde}.
Define the linear map
$\mu = \mu_{(A,\Phi)} : \Omega^1(\fsu(\cE,h))^2 \to
\Omega^2(\fsu(\cE,h))$ by
\begin{equation} \mu(\dot{A}, \dot{\Phi}) = \de_D \star \dot A - [\dot\Phi, \star
\Phi].\end{equation}
A variation $(\dot A, \dot \Phi)$ is $L^2$-orthogonal to the image of
the linearized gauge map $\rho$ if and only if it satisfies
$\mu(\dot{A}, \dot{\Phi}) = 0$. We say that such a variation is \ti{in
  gauge}.

For a general variation $(\dot A, \dot \Phi)$, the
orthogonal decomposition of \eqref{eq:orthog-decomp} is given by
\begin{equation} (\dot A, \dot \Phi)^\perp = (\dot A, \dot \Phi) - \rho(X)\end{equation}
where $X \in \Omega^0(\fsu(E))$ satisfies $\mu(\rho(X)) =
\mu(\dot{A}, \dot{\Phi})$.

For the specific variation obtained in \eqref{eq:unitary-gauge-dot}
we find that $\de_D \star \dot{A} = 0$, and a straightforward calculation yields
\begin{equation}
\begin{split}
Q := \mu(\dot A, \dot \Phi) &= -[\dot \Phi, \star \Phi]\\
& = -2
([\varphi_z, \dot\varphi^\dagger_{\bar z}] + [\varphi^\dagger_{\bar
  z}, \dot\varphi_z]) \de x \de y\\
&= -2 \e^{-2u} (P \dot{\overline{P}} - \overline{P} \dot{P}) \begin{pmatrix} 1 & 0 \\ 0 & -1 \end{pmatrix} \de x \de y.
\end{split}
\end{equation}
The computation of $(\dot A, \dot \Phi)^\perp$ therefore reduces to
solving
\begin{equation}
\mu(\rho(X)) = - \de_D \star \de_D X - [ [ X,\Phi], \star \Phi] =
Q
\label{eq:make-in-gauge}
\end{equation}
for $X$. Equation \eqref{eq:make-in-gauge} implies in particular that $X$ is diagonal and traceless;
thus we may write
\begin{equation}
\label{eq:def-X}
  X = \half \I v \begin{pmatrix} 1 & 0 \\ 0 & -1 \end{pmatrix}.
\end{equation}
After so doing, \eqref{eq:make-in-gauge} becomes a scalar equation for $v$,
\begin{equation} \label{eq:diff-v}
\Delta v - 8v(\e^{2 u} + \e^{-2u} \abs{P}^2) + 8 \e^{-2u} \im (P \dot{\overline{P}} ) = 0.
\end{equation}
We note the striking similarity between \eqref{eq:diff-v} and
\eqref{eq:diff-udot}; in fact, replacing $\dot{P} \to \I \dot{P}$
and $v \to -\dot{u}$ in \eqref{eq:diff-v} gives exactly
\eqref{eq:diff-udot}.  This suggests that we combine $\dot{u}$ (the
metric variation) and $v$ (the infinitesimal gauge transformation to
put the tangent vector in gauge) into the single complex function
\begin{equation}
F = \dot{u} - \I v,
\end{equation}
which we call the \emph{complex variation}, which then satisfies the
inhomogeneous linear equation
\begin{equation} \label{eq:diff-F}
\left(\Delta - 8(\e^{2 u} + \e^{-2u} \abs{P}^2)\right) F + 8 \e^{-2u} \overline{P} \dot{{P}} = 0.
\end{equation}
As with \eqref{eq:diff-udot} above, when working on the entire compact
surface $C$ the equation \eqref{eq:diff-F} uniquely determines the
complex function $F$. We write $F(\phi,\dot{\phi})$ for this unique
global solution determined by $(\phi,\dot{\phi}) \in T_\phi \cB$ when it
is necessary to distinguish it from other local solutions of the same
equation.

\subsection{Calculating the norm}

Using the calculations above we can now determine an explicit integral
expression for $g_\phi(\dot{\phi},\dot{\phi})$ in terms of $P$,
$\dot{P}$, $u$, and $F$.

The first step is to calculate $\rho(X)$ in unitary gauge.  We find $\rho(X) = (B,\Psi)$ where
\begin{equation}
\label{eq:rho-X}
 B =\frac{\I}{2} \begin{pmatrix} \de v & 0 \\ 0 & -\de
v \end{pmatrix}, \qquad
\Psi = \psi - \psi^\dagger, \qquad
\psi =  \I v \begin{pmatrix} 0 & \e^{-u} P \\ \e^u & 0 \end{pmatrix}
\de z.
\end{equation}
Now $(\dot A , \dot \Phi)^\perp = (\dot A, \dot \Phi) - (B,\Psi)$ is
orthogonal to $(B,\Psi)$, hence the \hk norm of the associated tangent
vector to the moduli space $\cM$ is
\begin{equation}
g_\phi(\dot{\phi}, \dot{\phi}) = \| (\dot A, \dot \Phi)^\perp
\|^2 = \|(\dot A, \dot \Phi)\|^2 - \|(B, \Psi)\|^2.
\end{equation}
Now we need only to substitute the expressions for $(\dot A,\dot \Phi)$
from \eqref{eq:unitary-gauge-dot} and $(B,\Psi)$ from \eqref{eq:rho-X}
and simplify. Two observations will be useful in doing this.  First,
if $\Xi = \xi - \xi^\dagger \in \Omega^1(\fsu(\cE))$ where $\xi$ is
expressed in unitary gauge as $\xi = f(z) \de z$ for $f$ a
matrix-valued function, then
\begin{equation}
\label{eq:off-diag-local}
 |\Xi|^2 = 4 \tr(f \bar f^t) \, \de x \de y.
\end{equation}
Second, if $\beta = \begin{pmatrix}
\theta & 0\\
0 & -\theta 
\end{pmatrix}$, with $\theta$ a scalar $1$-form, then we have 
\begin{equation}
\label{eq:on-diag-local}
|\beta|^2 = | \!\star\!\beta |^2 = 2 \theta \wedge \star \theta.
\end{equation}
Using \eqref{eq:off-diag-local} to simplify $|\dot \Phi|^2$ and
\eqref{eq:on-diag-local} to simplify $|\dot A|^2$, we find%
\footnote{Abusing notation, we often write integrals over $C$
  with the integrand expressed in a \emph{local} coordinate and frame for
  $\cE$.}
\begin{equation} \label{eq:part-norm}
 \|(\dot A, \dot \Phi)\|^2 = \int_C |\dot A|^2 + |\dot\Phi|^2 =
\frac12 \norm{\de \dot u}^2 + \int_C 4 \left(  \e^{2u} \dot u^2 + \e^{-2u} \abs{P \dot{u} -
  \dot{P}}^2 \right) \de x \de y,
\end{equation}
where $\norm{\theta}^2 = \int_C \theta \wedge \star
\theta$. Proceeding similarly for
$\|(B,\Psi)\|^2$ using \eqref{eq:rho-X}, we have
\begin{equation} \label{eq:gauge-part}
 \|(B,\Psi)\|^2 = \half \| \de v \|^2 + \int_C \left( 4 (\e^{2u} + \e^{-2u} \abs{P}^2) v^2 \right) \de x \de y.
\end{equation}
Subtracting \eqref{eq:gauge-part} from \eqref{eq:part-norm} we obtain
\begin{equation}
g_\phi(\dot{\phi}, \dot{\phi}) = \int_C \left( 4 \e^{2u} \dot u^2 + 4 \e^{-2u} \abs{P \dot{u} - \dot{P}}^2 - 4 (\e^{2u} + \e^{-2u} \abs{P}^2) v^2\right) \de x \de y + \frac12 \norm{\de \dot u}^2  - \half \norm{\de v}^2.
\end{equation}
Next we integrate by parts on $C$ to replace $\norm{\de \dot u}^2$ and $\norm{\de v}^2$
by $- \int_C (\dot u \Delta \dot u) \de x \de y$ and $- \int_C (v \Delta v) \de x \de y$ respectively,
and substitute for $\Delta \dot u$ and $\Delta v$ using the differential equations \eqref{eq:diff-u}
and \eqref{eq:diff-v}.
A few terms cancel and we are left with
\begin{equation}
g_\phi(\dot \phi, \dot \phi) = \int_C \left( 4 \e^{-2 u} \abs{\dot P}^2 - 4 \e^{-2u} \dot u \re (P \dot{\overline{P}}) - 4 \e^{-2u} v \im (P \dot{\overline P}) \right) \de x \de y
\end{equation}
or more compactly,
\begin{equation}
\label{eq:l2norm-combined}
g_\phi(\dot{\phi}, \dot{\phi}) = \int_C 4 \e^{-2u} \left(\abs{\dot P}^2 - \re (F P \dot{\overline{P}})\right) \de x \de y.
\end{equation}
As a reassuring consistency check, 
note that $g_\phi$ is indeed a \ti{Hermitian} metric, i.e.
$g_\phi(\I \dot\phi, \I \dot\phi) = g_\phi(\dot\phi, \dot\phi)$:
one sees this easily from \eqref{eq:l2norm-combined}, since changing
$\dot\phi \to \I \dot\phi$ leads to $F \to \I F$ and $\dot P \to \I \dot P$.
The same is not true of \eqref{eq:part-norm} by itself: it holds only
once we subtract the pure gauge part \eqref{eq:gauge-part}.

To sum up the results of this section, and restate the formula
\eqref{eq:gsf} for $g^{\sf}$ in the same local coordinates, we have:
\begin{thm}
\label{thm:g-integral}
For any quadratic differential $\phi \in \cB'$ and tangent vector
$\dot{\phi} \in T_\phi \cB' = \cB$ with respective coordinate
expressions $\phi = P(z) \, \de z^2$ and $\dot \phi = \dot P(z) \, \de z^2$,
the norm of $\dot\phi$ in the \hk metric $g$ is given by
\eqref{eq:l2norm-combined}, where $u$ and $F$ are the solutions
of \eqref{eq:diff-u} and \eqref{eq:diff-F}.  The norm of the same tangent vector 
in the semiflat metric $g^\sf$ is
\begin{equation}
\label{eq:l2sfnorm-combined}
g_\phi^{\sf}(\dot{\phi}, \dot{\phi}) = \int_C 2 |P|^{-1} \abs{\dot
  P}^2 \, \de x \de y.
\end{equation}
\noproof
\end{thm}

The goal of the next three sections is to gain some control over the
integral expressions \eqref{eq:l2norm-combined} and
\eqref{eq:l2sfnorm-combined} by studying the behavior of the functions
$u$ and $F$.  We will see that these functions are
well-approximated by
\begin{equation}
\begin{split}
u &\approx \half \log \abs{P},\\
F &\approx \half \frac{\dot P}{P},
\end{split}
\end{equation}
at points that are not too close to the zeros of $\phi$.  It is easy to
check that substituting these approximations directly into
\eqref{eq:l2norm-combined} yields exactly the semiflat integral
\eqref{eq:l2sfnorm-combined}.  Bounding the difference $g - g^\sf$ thus
reduces to understanding the error in the approximations.

\section{Exponential decay principle}
\label{sec:decay}

We now develop a criterion for solutions to certain elliptic PDE on
regions in the plane to decay exponentially fast as we move away from
the boundary of the region.  The method is standard---combining the
maximum principle with the known behavior of the eigenfunctions of the
Laplacian---and the results in this section are surely not new.
A similar method was used in \cite{Minsky92}, for example, to derive the
exponential decay results for \eqref{eq:diff-u} that we will generalize in
\S \ref{sec:u-estimates}.

\begin{thm}
\label{thm:exp-decay}
Let $\Omega = \{ |z| < R \}$ be a disk in $\C$, and for $z \in \Omega$
let $\rho(z) = d(z, \partial \Omega) = R-|z|$ denote the distance to
the boundary of this disk.  Suppose that
$w \in C^2(\Omega) \cap C^0(\bar{\Omega})$ satisfies
\begin{equation}	\label{eq:w-eq}
(\Delta - k^2) w = g
\end{equation}
where $k,g \in C^0(\bar{\Omega})$, $k \geq 4$, and suppose that for every
$\gamma < 4$ there exists a constant $A(\gamma)$ such that $g$ obeys
the exponential decay condition
\begin{equation}
|g| < A(\gamma) \e^{- \gamma \rho}.
\end{equation}
Then, for any $\gamma < 4$, there exist constants $K(\gamma)$ and
$A'(\gamma)$, such that $w$ obeys the exponential decay condition
\begin{equation}
\label{eq:exp-decay-lemma-target}
|w| < K(\gamma) (M+A'(\gamma)) \e^{- \gamma \rho},
\end{equation}
where $M = \sup_{\bdry \Omega} |w|$.
Moreover, given any $\gamma' > \gamma$, $A'(\gamma)$ can be chosen to be equal
to $A(\gamma')$.
\end{thm}

The proof will rely on the following maximum principle.
\begin{lem}
\label{lem:max-principle}
Let $\Omega$ be a bounded region in $\C$, and let $w,v \in C^2(\Omega)
\cap C^0(\bar{\Omega})$ with $v \geq 0$.  Suppose $w$ satisfies
\begin{equation} \label{eq:pde-1-max}
(\Delta - k^2)w = g,
\end{equation}
and that $v$ satisfies
\begin{equation} \label{eq:pde-2-max}
(\Delta - \underline{k}^2)v = -\bar{g},
\end{equation}
where $k,g,\underline{k},\bar{g} \in C^0(\bar{\Omega})$ are functions
such that
\begin{equation}
k \geq \underline{k} > 0 \; \text{ and } \; |g| \leq \bar{g}.
\end{equation}
If $|w| \leq v$ on $\bdry \Omega$, then $|w| \leq v$ on
$\bar{\Omega}$.
\end{lem}

\begin{proof}
First we claim $w \leq v$, or equivalently that $v - w \geq 0$ on
$\Omega$.  By compactness of $\bar \Omega$, the function $v-w$
achieves its minimum at a point $p$, and it suffices to show that the
minimum value is nonnegative.  If $p \in \bdry \Omega$ then this is
true by the hypothesis that $|w| \leq v$ on $\bdry \Omega$.  If
$w(p) \leq 0$ then $(v-w)(p) \geq 0$ because $v$ is everywhere
nonnegative.  Thus the remaining case is that $w(p) > 0$ and $p$ is an
interior local minimum of $v-w$, hence $\Delta(v-w)(p) \geq 0$.  Then
we find
\begin{align*}
  0 \leq \Delta(v-w)(p) &= \underline{k}(p)^2 v(p)  - k(p)^2 w(p) - \bar{g}(p)- g(p) & \text{using \eqref{eq:pde-1-max}, \eqref{eq:pde-2-max}}\\
    &\leq \underline{k}^2(p) (v-w)(p) - \bar{g}(p) - g(p) &
                                                         \text{because }w(p) >
                                            0\text{ and }k \geq
                                            \underline{k}\\
    &\leq \underline{k}^2(p) (v-w)(p) & \text{because }\bar{g} + g \geq
                           0
\end{align*}
Since $\underline{k} > 0$ this shows $(v-w)(p) \geq 0$ as required.

To complete the proof we must also show that $-w \leq v$.  However,
this follows by applying the argument above to the function $w' = -w$,
which satisfies $(\Delta - k^2)w' = g'$, where $g' = -g$.  Since
$|w'| = |w| \leq v$ on $\bdry \Omega$ and
$|g'| = |g| \leq \bar{g}$, the necessary hypotheses still hold in this
case.
\end{proof}

In the proof of \autoref{thm:exp-decay} we will use
\autoref{lem:max-principle} to reduce to the case where $k$ is
constant and where $g$ and $w$ are both radially symmetric
eigenfunctions of the Laplacian.  In preparation for doing this, we
recall the properties of those eigenfunctions and relate them to the
exponential decay behavior under consideration.

The \emph{modified Bessel function of the first kind} $I_0$ is the
unique positive, even, smooth function on $\R$ such that
\begin{equation}
 \Delta I_0(|z|) = I_0(|z|)
\end{equation}
and $I_0(0)=1$.  Thus the function $I_0(c |z|) / I_0(c R)$ is the
solution to the Dirichlet problem for $(\Delta - c^2)$ on the disk
$|z| < R$ with unit boundary values.

The function $I_0(x)$ satisfies (see e.g.~\cite[Section~9.7.1]{AS64})
\begin{equation}
I_0(x) \sim (2 \pi x)^{-\half} \e^x,
\end{equation}
where $f \sim g$ means that $f(x)/g(x) \to 1$ as $x \to \infty$. It
follows that, if a function $f$ satisfies an exponential decay condition
\begin{equation}
 	\label{eq:decay-cond-1}
 	f < A \e^{- \gamma \rho}
 \end{equation} 
for some $\gamma > 0$,
then for any $\tilde\gamma < \gamma$ we have
\begin{equation}
 	\label{eq:decay-cond-2}
 f < \tilde{A} \, \frac{I_0(\tilde{\gamma} |z|)}{I_0(\tilde{\gamma} R)}
\end{equation}
for some $\tilde{A}(A,\gamma,\tilde\gamma)$ linear in $A$.
Conversely, if we have \eqref{eq:decay-cond-2} and $\gamma \le \tilde\gamma$
then we get \eqref{eq:decay-cond-1} for some $A(\tilde A,\gamma,\tilde\gamma)$
linear in $\tilde A$.\footnote{Adjusting the
constants when converting between exponential and Bessel bounds is
necessary due to the $x^{-\half}$ factor in
the expansion of $I_0$, with the relevant observation being that
$x^{\half} \e^{-\gamma x} = O(\e^{-\tilde{\gamma} x})$ for all
$\tilde{\gamma} < \gamma$ whereas of course
$x^{\half} \e^{-\gamma x} \neq O(\e^{-\gamma x})$.}

\begin{proof}[Proof of \autoref{thm:exp-decay}]
Suppose we are given constants $\gamma < \gamma' < 4$.
The function $g$ obeys
\begin{equation}
	\abs{g} < A(\gamma') \e^{-\gamma' \rho}
\end{equation}
and thus
\begin{equation}
\label{eqn:g-assumption}
\abs{g} < \tilde A \frac{I_0( \gamma |z|)}{I_0(\gamma R)}
\end{equation}
for some $\tilde A(A, \gamma, \gamma')$ linear in $A = A(\gamma')$.
For $w$ satisfying \eqref{eq:w-eq} and $g$ satisfying \eqref{eqn:g-assumption}
we will show that
\begin{equation}
\label{eqn:w-goal}
|w| < \tilde{K}(M+\tilde A) \frac{I_0( \gamma |z|)}{I_0(\gamma R)}
\end{equation}
for some $\tilde{K}(\gamma)$.
Once this is achieved we can pass back from \eqref{eqn:w-goal} to 
the desired \eqref{eq:exp-decay-lemma-target} using the exponential
bound on $I_0$ discussed above.
Moreover, since $\tilde A$ depends linearly on $A(\gamma')$, we can choose
$A' = A$ in \eqref{eq:exp-decay-lemma-target}, at the cost of possibly
rescaling $\tilde K$.

Define 
\begin{equation}
v = B \, I_0(\gamma |z|)
\end{equation}
for a constant $B > 0$.
Note that $(\Delta - \gamma^2)v = 0$.  We will determine a value of
$B$ so that \autoref{lem:max-principle} can be applied to $v$ and
$w$ on $\Omega$.  Specifically, we must ensure that:
\begin{rmenumerate}
\item \label{item:bdry} $|w| \leq v$ on $\bdry \Omega$ and
\item \label{item:inhomog} $(\Delta - 16)v \leq -|g|$ on $\Omega$,
\end{rmenumerate}
so that, in the notation of \autoref{lem:max-principle}, we can take
$\bar{g} = -(\Delta - 16)v$ and $\underline{k} = 4$. 

First we consider \ref{item:bdry}.  The function $v$
is constant on $\bdry
\Omega$ and equal to $B \,I_0(\gamma R)$.  Since $M = \sup_{\bdry \Omega}
|w|$, it suffices to choose
\begin{equation}
\label{eqn:sufficient-for-bdry}
 B \geq \frac{M}{I_0(\gamma R)}.
\end{equation}

Now we turn to \ref{item:inhomog}.  We have 
\begin{equation}
(\Delta - 16)v = - (16 - \gamma^2)v
\end{equation}
which we have written in this way to emphasize that
$(16 - \gamma^2) > 0$.  With the given bound \eqref{eqn:g-assumption} on
$|g|$, the desired inequality \ref{item:inhomog} follows if
\begin{equation}
B \, I_0(\gamma |z|) (16 - \gamma^2) \geq \tilde A \frac{I_0(\gamma |z|)}{ I_0(\gamma R)},
\end{equation}
or equivalently
\begin{equation} \label{eqn:sufficient-for-inhomog}
 B \geq  \frac{\tilde A}{16 - \gamma^2} \frac{1}{I_0(\gamma R)}.
\end{equation}

Using \eqref{eqn:sufficient-for-bdry} and
\eqref{eqn:sufficient-for-inhomog} it is easy to verify that
\begin{equation}
B = \max \left ( 1, \frac{1}{16-\gamma^2} \right ) \cdot (M+\tilde A) \cdot
\frac{1}{I_0(\gamma R)}
\end{equation}
satisfies both conditions \ref{item:bdry}-\ref{item:inhomog},
and then by \autoref{lem:max-principle} we find $|w| \leq v$ on $D$, which is
the desired bound \eqref{eqn:w-goal} with
\begin{equation}
\tilde K = \max \left ( 1, \frac{1}{16 - \gamma^2} \right ).
\end{equation}
\end{proof}

\section{Estimates for the density \texorpdfstring{$u$}{u}}
\label{sec:u-estimates}

As in Section \ref{sec:coordinate-computations} above, let
$u = u(\phi)$ be the solution of the self-duality equation
\eqref{eq:diff-u} on the compact Riemann surface $C$ for a given
holomorphic quadratic differential $\phi \in \cB$ expressed locally
as $\phi = P(z) \, \de z^2$.  It was shown by Minsky in \cite{Minsky92}
(see also \cite[Lemma~2.2]{Wolf91}) that $u$ is approximated by
$\half \log |P|$ up to an error that decays exponentially in the
distance from the zeros of $\phi$.  Building on Minsky's results (and
following a similar outline to \cite[Section~5.4]{DW15}), we establish
the following estimate which gives a slightly faster exponential decay
rate:

\begin{thm}
\label{thm:u-exp}
Fix $\phi \in \cB$ and assume $M(\phi) > 1$.
For any $\gamma < 4$, there exist constants $A(\gamma)$ and $b(\gamma)$ 
such that the density $u = u(\phi)$ satisfies
\begin{equation}
\label{eq:u-exp-c0}
 \left | u(z) - \half \log |P(z)| \right | < A(\gamma) \e^{-\gamma r(z) }  
\end{equation}
for all $z \in C$ with $r(z) > b(\gamma)$.  The constants $A(\gamma)$
and $b(\gamma)$ can be taken to depend only on $\gamma$ and the
topological type of $C$.

Furthermore, under the same hypotheses we have the $C^1$ estimate
\begin{equation}
\label{eq:u-exp-c1}
\left | \nabla_\phi (u - \half \log |P| )(z) \right |_\phi < A(\gamma) \e^{-\gamma r(z) }
\end{equation}
where $\nabla_\phi$ and $|v|_\phi$ denote, respectively, the gradient
and the norm of a tangent vector with respect to the metric $|\phi|$.
\end{thm}

To prove this, we will first establish some rough bounds on $u$.
These will allow us to apply \autoref{thm:exp-decay} to the
equation satisfied by $u - \half \log |P|$.

\subsection{Rough bounds}

Let $\e^{2 \sigma} \abs{\de z}^2$ be the Poincar\'e metric on $C$ of
constant (Gaussian) curvature $-4$.  In general, the Gaussian
curvature of a metric $\e^{2 u} \abs{\de z}^2$ is given by
$K=-\e^{-2u} \Delta u$ (see e.g.~\cite[Section~1.5]{ahlfors73});
therefore, the equation $K=-4$ satisfied by the Poincar\'e metric becomes
\begin{equation} \label{eq:diff-sigma}
	\Delta \sigma = 4 \e^{2 \sigma},
\end{equation}
which is equation \eqref{eq:diff-u} with $\phi=0$.

Now for the solution $u = u(\phi)$ of \eqref{eq:diff-u} associated to
a general quadratic differential $\phi = P \,\de z^2$, we have the
following lower bounds in terms of $\sigma$ and $P$:
\begin{lem}
\label{lem:u-rough-lower}
We have $u - \sigma \ge 0$ everywhere on $C$, and $u - \half \log |P|
\geq 0$ on $C \setminus \phi^{-1}(0)$.
\end{lem}

\begin{proof}
Using \eqref{eq:diff-u} and \eqref{eq:diff-sigma} we have
\begin{equation} \label{eq:delta-difference}
	\Delta(u - \sigma) = 4(\e^{2u} - \e^{2\sigma} - \e^{-2u}
        \abs{P}^2) \leq 4(\e^{2u} - \e^{2\sigma}).
\end{equation}
At a minimum of $u - \sigma$, we have
$\Delta(u - \sigma) \geq 0$ and so by \eqref{eq:delta-difference} we find
$\e^{2u} - \e^{2\sigma} \geq 0$ there.  Thus $u - \sigma \ge 0$ at the
minimum, hence everywhere on $C$.

The lower bound on $u - \half \log |P|$ is similar. Using that $\half
\log|P|$ is harmonic on $C \setminus \phi^{-1}(0)$, we find that
\begin{equation}
\label{eq:u-difference-eqn}
\Delta (u - \half \log|P|) = 4(\e^{2u} - \e^{-2u}
        \abs{P}^2).
\end{equation}
Since $\half \log|P(z)| \to -\infty$ as $z$ approaches a zero of
$\phi$, while $u$ is smooth on the entire surface $C$, the difference
$u - \half \log|P|$ has a minimum on $C \setminus \phi^{-1}(0)$.  At
such a minimum we have $\Delta (u - \half \log|P|) \geq 0$, which
gives $\e^{2u} - \e^{-2u} \abs{P}^2 \geq 0$.  This implies
$u - \half \log|P| \geq 0$ at a minimum, and thus everywhere.
\end{proof}

Complementing these lower bounds on $u$, we have the following rough
comparison to the singular flat metric $|\phi|$.  Recall that the
radius $r(z)$ is the distance from $z$ to $\phi^{-1}(0)$ with respect
to the metric $|\phi|$.

\begin{lem}[Minsky {\cite[Lemma 3.2]{Minsky92}}]
\label{lem:u-rough-upper}
Fix $\phi \in \cB$ and assume $M(\phi) > 1$.  Let $z \in C$ be a point
with $r(z) \geq 1$.  Then $u(z) - \half \log |P(z)| \leq M$ where $M$
is a constant depending only on the topological type of $C$. \noproof
\end{lem}

Note that Minsky's bound is more general, giving an upper bound at any
point $z$ depending only on the topological type of $C$ and on the
$|\phi|$-radius $R$ of an embedded disk centered at $z$ that contains
no zeros of $\phi$.  The hypotheses of the lemma above give such a
disk of definite radius (in fact, one can take $R = \half$), resulting
in the bound stated above that only depends on the topological type.

\subsection{Exponential bounds}
\begin{proof}[Proof of \autoref{thm:u-exp}]
We start with the $C^0$ bound \eqref{eq:u-exp-c0}.  Consider a local
coordinate $\zeta$ about $z$ in which $\phi = \de \zeta^2$.  Allowing
this coordinate chart to be immersed, rather than embedded, we can
take it to be defined on $|\zeta| < r(z)$ with $z$ corresponding to
$\zeta=0$.  While the boundary of this disk touches the zero set of
$\phi$ (by definition of $r$), if we consider
$D = \{ |\zeta| < r(z) -1 \}$ then the image of this disk in $C$
consists of points satisfying the hypotheses of
\autoref{lem:u-rough-upper}.  Therefore, by \autoref{lem:u-rough-lower}
and \autoref{lem:u-rough-upper} we have $0 \leq u(\zeta) \leq M$ for all
$\zeta \in D$.

In this coordinate system we have $P(\zeta) \ident 1$ and thus
\eqref{eq:u-difference-eqn} becomes
\begin{equation}
\label{eq:u-difference-eqn-P1}
\Delta u = 4 \e^{2u} - 4 \e^{-2u} = 8
\sinh(2u).
\end{equation}
Similarly, in this coordinate the difference $|u - \half
\log |P||$ reduces to $|u|$.

The function $8\sinh(2x)/x$ has a removable singularity at $x=0$; let
$f$ denote its extension to a smooth function on $\R$, which satisfies
$f(x) \geq 16$ for $x \geq 0$.  Since $u \geq 0$ we can rewrite the
equation above as
\begin{equation} (\Delta - k^2)u = 0\end{equation}
where $k = \sqrt{f(u)}$, and thus $k \geq 4$.  Now
\autoref{thm:exp-decay} applies to $u$ on $D$ with $g \ident 0$
and $\rho = r(z)-1$, giving
\begin{equation} |u| \leq A(\gamma) \e^{-\gamma (r(z)-1)
} \end{equation} for all $\gamma < 4$.  Absorbing the $\e^{\gamma}$
factor into the multiplicative constant we obtain the desired bound
\eqref{eq:u-exp-c0} in terms of $\e^{-\gamma r(z)}$.

Given this $C^0$ bound, the corresponding $C^1$ bound
\eqref{eq:u-exp-c1} follows by standard elliptic theory applied to
\eqref{eq:u-difference-eqn-P1}, as shown in e.g.~\cite[Corollary~5.10]{DW15}.
\end{proof}

\section{Estimates for the complex variation \texorpdfstring{$F$}{F}}
\label{sec:F-estimates}

Next we turn to the complex variation $F = F(\phi,\dot\phi)$ associated
to $\phi \in \cB'$ and $\dot{\phi} \in T_\phi \cB$.  We will see that
this function is exponentially close to
$\half \frac{\dot{\phi}}{\phi} = \half \frac{\dot{P}}{P}$.
Specifically, we have:

\begin{thm}
\label{thm:f-exp}
Fix $\phi \in \cB$ and assume $M(\phi) > 1$. Also fix $\dot\phi$.
For any $\gamma < 4$, there exist constants $A(\gamma)$ and $b(\gamma)$
such that the function $F(\phi, \dot\phi)$ satisfies
\begin{equation}
 \left | F - \frac{1}{2} \frac{\dot{\phi}}{\phi} \right | < A(\gamma)
\|\dot{\phi}\| \e^{-\gamma r(z)}
\end{equation}
for all $z \in C$ with $r(z) > b(\gamma)$.
\end{thm}

To prove this, we proceed as in \S \ref{sec:u-estimates}, first deriving
some rough bounds, and then improving them to exponential bounds using
\autoref{thm:exp-decay}.

\subsection{Rough bounds}
We begin with some notation related to metrics on $C$.  If $\eta$ is a
log density on $C$, with associated K\"ahler metric
$\e^{2 \eta} |\de z^2|$, and if $\phi \in \cB$ has local expression
$\phi = P \,\de z^2$, we denote by
\begin{equation} |\phi|_\eta = \e^{-2 \eta} |P| : C \to
\R \end{equation} the pointwise norm function of $\phi$ with respect
to this metric, and by
\begin{equation}\|\phi\|_\eta = \sup_C |\phi|_\eta\end{equation} the
associated sup-norm.  Finally, we let
\begin{equation} \Delta_\eta = \e^{-2 \eta} \Delta \end{equation}
denote the Laplace-Beltrami operator of the metric $\e^{2 \eta} |\de z^2|$.

Recall from \S \ref{sec:u-estimates} that
$\sigma$ denotes the density of the Poincar\'e metric on $C$ of curvature $-4$.

\begin{lem}
\label{lem:F-rough}
The complex variation $F$ satisfies $\sup |F| \leq
\|\dot{\phi}\|_\sigma$.
\end{lem}

\begin{proof}
Rewriting \eqref{eq:diff-F} in terms of the Laplace-Beltrami operator
$\Delta_u$ it becomes:
\begin{equation}
\label{eq:diff-F-named-coefs}
(\Delta_u - K) F = -G,
\end{equation}
where
\begin{equation}
\begin{split}
K &= 8 (1 + |\phi|_u^2), \\
G &= \frac{8\bar{\phi} \dot{\phi}}{\e^{4u} |\de z|^4}.
\end{split}
\end{equation}
Note that $K > 0$ and that $G$ is a well-defined complex scalar
function on $C$ which satisfies
\begin{equation}
 |G| = 8 |\phi|_u |\dot{\phi}|_u.
\end{equation}
Considering a maximum and a minimum of each of the real and imaginary parts of $F$, 
which exist by compactness, we find from \eqref{eq:diff-F-named-coefs} that
\begin{equation}
 \sup |F| \leq \sup 2 K^{-1} |G|,
\end{equation}
and we have
\begin{equation}
 K^{-1}|G| = |\dot{\phi}|_u \left ( \frac{|\phi|_u}{1 + |\phi|_u^2} \right
) \leq \frac{1}{2} |\dot{\phi}|_u.
\end{equation}
Finally, by \autoref{lem:u-rough-lower}, we have
$u \geq \sigma$.  Therefore $|\dot{\phi}|_u \leq |\dot{\phi}|_\sigma$ and
\begin{equation}
\sup |F| \leq \sup 2 K^{-1}|G| \leq \sup |\dot{\phi}|_\sigma =
\|\dot{\phi}\|_\sigma. 
\end{equation}
\end{proof}

We will also need the following lower bound on the pointwise norm
$|\phi|_\sigma$.  Recall $r(z)$ denotes the $|\phi|$-distance from $z$
to $\phi^{-1}(0)$.

\begin{lem}
\label{lem:phi-pointwise}
Let $\phi \in \cB'$ and suppose $\|\phi\|_\sigma \geq 1$.  There
exists a constant $\delta$ depending only on the ray $\R_+ \phi$ with
the following property: If $z \in C$ satisfies $r(z) > 1$,
then $|\phi|_\sigma(z) \geq \delta$.
\end{lem}

\begin{proof}
Let $Z = \phi^{-1}(0)$.  First suppose that $\|\phi\|_\sigma = 1$.
For any positive radius $r_0$, a uniform lower bound on
$|\phi|_\sigma(z)$ for $z$ with $r(z) \geq r_0$ follows
immediately from compactness of $C$ and of the unit ball in
$\cB$.

Using that $d_{t \phi} = t^{1/2} d_\phi$, the same argument shows that
for $\|\phi\|_\sigma \geq 1$ we also have a uniform lower bound on
$|\phi|_\sigma(z)$ when $r(z)$ is greater than a fixed
positive multiple of $\|\phi\|_\sigma^{1/2}$.

Thus to complete the argument it suffices to consider the case when
$r(z)$ is small compared to $\|\phi\|_\sigma^{1/2}$.  That
is, we consider a point $z$ in a disk of radius
$\epsilon \|\phi\|_\sigma^{1/2}$ about one of the zeros.
Equivalently, if we let $\phi_0 = \|\phi\|_\sigma^{-1} \phi$, then $z$
lies in a disk of $\phi_0$-radius $\epsilon$ about a zero of $\phi_0$.
Assume that $\epsilon$ is small enough (depending on the ray) so
that there is only one zero of $\phi_0$ in this disk, and that the
disk is identified with $|\zeta| < R$ by a coordinate function $\zeta$ such that
$\phi_0 = \zeta \de \zeta^2$. We work in this coordinate system for
the rest of the proof.

Write the Poincar\'e metric of $C$ on this disk as
$\e^{2\sigma} |\de \zeta|^2$.  Then, using compactness of the unit
ball in $\cB$ again, we have $\e^{2\sigma} \leq M$ for a uniform constant
$M$.

The $\phi_0$-distance from $0$ to $\zeta$ is proportional to
$|\zeta|^{3/2}$, and thus the $\phi$-distance is proportional to
$\|\phi\|_\sigma^{1/2}\, |\zeta|^{3/2}$, with universal constants in
both cases.  The hypothesis that $r(z) \geq 1$ therefore becomes
$|\zeta(z)| \geq c \, \|\phi\|_\sigma^{-1/3}$ for a constant $c > 0$.
Using that $\phi = \|\phi\|_\sigma \zeta \de \zeta^2$, at such a point $z$
we have
\begin{equation} |\phi|_\sigma(z) = \e^{-2 \sigma(z)} \|\phi\|_\sigma
\, |\zeta(z)| \geq M^{-1}  \|\phi\|_\sigma  (c
\|\phi\|_\sigma^{-1/3} ) = M^{-1} c \|\phi\|_\sigma^{2/3}
\end{equation}
Since we assumed $\|\phi\|_\sigma \geq 1$, this gives the desired lower
bound with $\delta = M^{-1} c$.
\end{proof}

\subsection{Exponential bounds}

\begin{proof}[Proof of \autoref{thm:f-exp}]
Define
\begin{equation} f = F - \frac{1}{2} \frac{\dot{\phi}}{\phi},\end{equation}
so that our goal is to give an exponentially decaying upper bound on
$|f|$ at a point $z$. As in the proof of \autoref{thm:u-exp} we first choose an immersed
coordinate chart $|\zeta| < r(z)$ where $\phi = \de \zeta^2$ and $z$
corresponds to $\zeta=0$.

Using \eqref{eq:diff-F}, after a bit of algebra we find that in this
coordinate system $f$ satisfies the equation
\begin{equation}
\label{eq:diff-F-in-P-coord}
(\Delta - 16\cosh(2u)) f = -8 \frac{\dot{\phi}}{\phi} \sinh(2u).
\end{equation}

Let $\eps > 0$, and restrict attention to the smaller disk
$\Omega = \{ |\zeta| \leq (1-\eps) r(z) \}$.  Assume that
$r(z) > \eps^{-1}$.  Then all points of $\Omega$ are at distance at
least $1$ from $\phi^{-1}(0)$, and \autoref{lem:phi-pointwise} gives
\begin{equation}
\label{eqn:phidot-over-phi}
\left\lvert \frac{\dot{\phi}}{\phi} \right\rvert = \frac{|\dot{\phi}|_\sigma}{|\phi|_\sigma}
\leq \delta^{-1} \|\dot{\phi}\|_\sigma
\end{equation}
throughout $\Omega$.  Now fix $\gamma < 4$, let $b(\gamma)$ denote
the constant from \autoref{thm:u-exp}, and assume that $r(z) >
\eps^{-1} b(\gamma)$.  Then \autoref{thm:u-exp} applies to $u$ at each point
of $\Omega$, giving
\begin{equation}
 \abs{u(\zeta)} < A(\gamma) \e^{-\gamma r(\zeta) }.
\end{equation}
Combining this with the bound \eqref{eqn:phidot-over-phi} on
$\frac{\dot{\phi}}{\phi}$, we find that the right hand side of
\eqref{eq:diff-F-in-P-coord} is bounded above by
\begin{equation} \label{eq:interm-bound}
 A'(\gamma) \|\dot{\phi}\|_\sigma \e^{-\gamma r(\zeta)}
\end{equation}
for some constant $A'(\gamma)$.

Let $\rho$ denote the function on
$\Omega$ that gives the $\phi$-distance to $\partial \Omega$.
Since $r(z) > \rho(z)$ we can replace \eqref{eq:interm-bound}
by
\begin{equation}
 A'(\gamma) \|\dot{\phi}\|_\sigma \e^{-\gamma \rho}.
\end{equation}
When combined with the fact that $16 \cosh(2u) \geq 16$, this
exponential decay of the inhomogeneous term of
\eqref{eq:diff-F-in-P-coord} implies that the solution $f$ is also
exponentially decaying; specifically, applying
\autoref{thm:exp-decay} we have for any $\gamma < 4$
\begin{equation}
\label{eqn:f-almost-done}
|f| \leq K(\gamma) (M +  A''(\gamma) \|\dot{\phi}\|_\sigma) \e^{- \gamma \rho}
\end{equation}
where $M = \sup_{\bdry \Omega} |f|$. We have
$|f| \leq |F| + \half \left |\frac{\dot{\phi}}{\phi} \right |$,
and therefore \autoref{lem:F-rough} and \eqref{eqn:phidot-over-phi} 
give $M \leq (1 + \delta^{-1}) \|\dot{\phi}\|_\sigma$.
Substituting this value for $M$ into \eqref{eqn:f-almost-done} and
evaluating at $\zeta = 0$ (i.e.~at $z$), where $\rho = (1-\eps) r(z)$, we obtain
\begin{equation} |f(z)| \leq A'''(\gamma) \|\dot{\phi}\|_\sigma \e^{-\gamma(1-\eps) r(z)}\end{equation}
for a constant $A'''(\gamma)$ depending on the ray $\R_+ \phi$.
Since $\eps$ was arbitrary, this gives the desired bound.
\end{proof}

\section{Holomorphic variations}
\label{sec:holomorphic}

In our analysis of $C_\near$ we will exploit the following basic observation:
if $D \subset C$ is a disk containing exactly one zero of $\phi$,
then any holomorphic quadratic differential 
$\dot\phi$ on $D$ can be realized as $\dot\phi = \cL_X \phi$
for some holomorphic vector field $X$ on $D$.
This fact allows us to construct an explicit solution of the complex
variation equation \eqref{eq:diff-F} on $D$, using the following:

\begin{thm}
\label{thm:holo-variation}
Given a quadratic differential $\phi = P(z) \, \de z^2$, solution $u$ of
\eqref{eq:diff-u}, and holomorphic vector field $X = \chi(z)
\frac{\partial \:}{\partial z}$, let
\begin{equation}\label{eq:L-X-phi}
 \psi_X = \mathcal{L}_X \phi =  (\chi P_z + 2 \chi_z P) \,\de z^2\end{equation}
and define a complex scalar function $F_X$ by
\begin{equation}\label{eq:F-X}
F_X \e^{2 u} |\de z|^2 = \mathcal{L}_X (\e^{2u} |\de z|^2)\end{equation}
or equivalently in local coordinates 
\begin{equation} F_X = \chi_z + 2 \chi u_z. \end{equation}
Then $F = F_X$ satisfies the complex variation equation
\eqref{eq:diff-F} with $\dot{\phi} = \psi_X$.
\end{thm}

\begin{proof}
We begin by noting that the 
self-duality equation \eqref{eq:diff-u} is natural with respect to
biholomorphic maps, i.e.~if $\Phi$ is such a map then the log density of the pullback metric
$\Phi^*( \e^{2u} |\de z|^2)$ satisfies the equation for the pullback
differential $\Phi^* \phi$. The real vector field
$X + \bar{X}$ has a local flow which consists of holomorphic maps,
and hence gives rise to a local $1$-parameter family of solutions for
the corresponding family of pullback quadratic differentials.  Taking
the derivative of this family of solutions at $t=0$ we find that the
Lie derivative of $\e^{2u} |\de z|^2$ with respect to $X + \bar{X}$
gives a solution of the variation equation \eqref{eq:diff-udot} for
\begin{equation} \dot{\phi} = \mathcal{L}_{X+\bar{X}} \phi.\end{equation}
Specifically, if we define $\dot{u}$ by
\begin{equation}
 \dot{u} \e^{2 u} |\de z|^2 = \mathcal{L}_{X+\bar{X}}(\e^{2u} |\de z|^2)
\label{eq:udot-x}
\end{equation}
then $\dot{u}$ and $\dot{\phi}$ satisfy \eqref{eq:diff-udot}.  The
expression \eqref{eq:udot-x} is equivalent to saying that $\dot{u}$ is the
Riemannian divergence of the vector field $X + \bar{X}$ with respect
to the metric $\e^{2u} |\de z|^2$.

Now, recall that \eqref{eq:diff-F} is equivalent to the separate equations
\eqref{eq:diff-udot} for $\dot{u} = \re(F)$ and \eqref{eq:diff-v} for
$v = - \im(F)$, and that these two equations are related by the
substitutions $\dot{u} \to -v$ and $\dot{\phi} \to \I \dot{\phi}$.

For a real tensor $T$ we have
$\re(\mathcal{L}_X T) = \mathcal{L}_{X + \bar{X}} T$, and hence
$\re(F_X)$ is exactly $\dot{u}$ as defined by \eqref{eq:udot-x}, which
we have seen satisfies \eqref{eq:diff-udot} with
$\dot{\phi} = \mathcal{L}_{X+\bar{X}} \phi$.  Because $\phi$ is
holomorphic we in fact have
$\mathcal{L}_{X+\bar{X}} \phi = \mathcal{L}_X \phi = \psi_X$.  Hence
$\re(F_X)$ satisfies the desired equation.

Because $\mathcal{L}_{\I X} = \I \mathcal{L}_X$ we have
$-\im(F_X) = \re(F_{\I X})$ which therefore satisfies
\eqref{eq:diff-udot} with $\dot{\phi} = \mathcal{L}_{\I X+\bar{iX}} \phi =
\mathcal{L}_{\I X} \phi = \I \psi_X$.  Using the substitutions noted
above, this is equivalent to
$\im(F_X)$ satisfying \eqref{eq:diff-v}.
\end{proof}

\section{Exactness}
\label{sec:exactness}

Given quadratic differentials $\phi$ and $\dot{\phi}$ on a compact
Riemann surface $C$, recall that our ultimate goal is to bound the difference
\begin{equation}\Delta(\phi, \dot{\phi}) := g_\phi(\dot{\phi},\dot{\phi}) -
g_\phi^{\sf}(\dot{\phi},\dot{\phi}).\end{equation}
Though the integrals defining $g_\phi$ and $g_\phi^{\sf}$ were
previously written in terms of densities (scalar multiples of
$\de x \de y$), using the orientation of $C$ we can convert the integrand
to a differential $2$-form which we denote by $\delta$.  Also recall
that this integrand depends on the density $u$ and complex function $F$,
respectively satisfying \eqref{eq:diff-u} and \eqref{eq:diff-F}.
Explicitly, by taking the difference of the integral expressions
\eqref{eq:l2norm-combined}-\eqref{eq:l2sfnorm-combined} we find
\begin{equation}
\label{eq:delta}
\delta(\phi,\dot\phi,u,F) = \left( 4\e^{-2u}( |\dot{P}|^2 - \re(F P
\dot{\overline{P}}) ) - 2 \frac{|\dot{P}|^2}{|P|} \right) \de x
\wedge \de y,
\end{equation}
where $\dot{\phi} = \dot{P} \, \de z^2$ and as usual $\phi = P \, \de z^2$.
Thus if $u(\phi)$
and $F(\phi,\dot\phi)$ denote the unique solutions to \eqref{eq:diff-u} and
\eqref{eq:diff-F} on a compact Riemann surface $C$ for given $\phi$
and $\dot\phi$, then we have
\begin{equation}
\Delta(\phi,\dot\phi) = \int_C \delta(\phi,\dot{\phi},u(\phi),F(\phi,\dot\phi)).
\end{equation}

As mentioned in \S\ref{sec:strategy}, our technique for bounding the
integral of $\delta(\phi,\dot{\phi},u(\phi),F(\phi,\dot\phi))$ over the
region $C_{\near}$ near the zeros of $\phi$ involves approximating
$\delta$ in that region by an exact form.  The key to this approximation
is that $\delta(\phi,\dot{\phi},u,F)$ itself is exact whenever
$\dot{\phi}$ and $F$ are obtained from $\phi$ and $u$ using a
holomorphic vector field as in \autoref{thm:holo-variation}:

\begin{lem}
\label{lem:exactness}
Let $\phi = P \, \de z^2$ be a quadratic differential and $u$ a log density satisfying
\eqref{eq:diff-u}, both on a domain $U \subset \C$. Let $X = \chi \partial_z$ be a
holomorphic vector field on $U$. Let 
$\dot\phi = \cL_X \phi$ and 
$F = F_X$ as in \autoref{thm:holo-variation}.
Then $\delta(\phi, \dot\phi, u, F) = \de \beta$, where
\begin{equation}
 \beta = \left ( \e^{-2u} - |P|^{-1} \right )\left ( 2 |P|^2 \star \de
|\chi|^2 + |\chi|^2 \star \de |P|^2 \right).
\end{equation}
\end{lem}

\begin{proof}
Substituting $\dot{\phi} = \mathcal{L}_X \phi$ as given by
\eqref{eq:L-X-phi} and $F = F_X$ from \eqref{eq:F-X} into \eqref{eq:delta}, we obtain an
explicit formula in terms of $\chi$, $P$, and $u$:
\begin{multline}
\label{eq:delta-for-holo}
\delta(\phi, \mathcal{L}_\chi \phi, u, F_X) = \Biggl ( 4 \e^{-2u}
\Bigl ( %
\abs{\chi P_z}^2 + 2 \abs{\chi_z P}^2
+ 3 \re( \chi \bar{\chi}_{\bar z} P_z \bar{P})
- 2 \re( \chi \bar{\chi} P \bar{P}_{\bar z} u_z )\\
\qquad\qquad
- 4 \re( \chi \bar\chi_{\bar z} P \bar{P} u_z )
\, \Bigr)
-2 \frac{\abs{\chi P_z}^2}{|P|}
- 8 \frac{\re(\chi \bar{\chi}_{\bar z} \bar{P} P_z)}{|P|}
- 8 \abs{\chi_z}^2 |P|
\Biggr )
\de x \wedge \de y.
\end{multline}
Now we consider $\beta$. For a holomorphic function $f$, we have
\begin{equation}
\star \de |f|^2 = \star \de (f \bar{f}) = \star \left ( f_z \bar{f} \de z + f
\bar{f}_{\bar z} \de \bar{z} \right ) = -\I \left ( f_z \bar{f} \de z - f
\bar{f}_{\bar z} \de \bar{z} \right ) = 2 \im ( f_z \bar{f} \de z).
\end{equation}
Using this, we find that $\beta = 2 \im(\tilde{\beta})$ where
\begin{equation}
\tilde{\beta} = (\e^{-2u} - |P|^{-1}) \left (2 |P|^2 \chi_z
\bar{\chi} + |\chi|^2 P_z\bar{P} \right ) \de z.
\end{equation}
It is then straightforward to calculate
$\de \beta = 2 \im( \bar{\partial} \tilde{\beta})$ in terms of $P$
and $\chi$, and to verify that it is equal to
\eqref{eq:delta-for-holo}; in the latter step, it is useful to recall
$\im(c \, \de \bar{z} \wedge \de z) = 2 \re(c)\, \de x \wedge \de y$ for any
complex scalar $c$.  We omit the details of this lengthy but
elementary calculation.
\end{proof}

\section{Exponential asymptotics}
\label{sec:final}

In this section we prove \autoref{thm:main}.  To do so we return to
considering a compact Riemann surface $C$ and the ray
$\{\phi = t \phi_0\}_{t \in \R_+}$ generated by $\phi_0 \in \cB'$.
Write $\phi = P \, \de z^2$.  Let
$\dot{\phi} = \dot{P} \, \de z^2 \in T_\phi \cB' = \cB$.  Fix some
$\gamma < 4$.

Let $z_1, \ldots, z_n$ denote the zeros of $\phi$, and let $D_i$ denote
an open disk centered on $z_i$ of $\abs{\phi}$-radius
\begin{equation}  \label{eq:Rtdep}
	R = \half M(\phi) = \half t^\half M(\phi_0).
\end{equation}
This is the largest $\phi$-radius for which the sets $D_i$ are disjoint,
embedded disks.  Note that $D_i$ can also be described as the disk about $z_i$ of
$\abs{\phi_0}$-radius $\half M(\phi_0)$, and in particular the set $D_i$ is
independent of $t$.

Since we are considering asymptotic statements as $t \to \infty$, and
since by hypothesis $M(\phi_0) > 0$, we may assume when necessary that
$R$ is larger than any given constant.

Let $C_{\near} = \bigcup_i D_i$ and $C_{\far} = C \setminus C_{\near}$.
Then we have
\begin{equation}
\label{eqn:decomposition}
\Delta(\phi,\dot{\phi}) = \int_{C_{\far}} \delta(\phi,\dot\phi,u(\phi),
F(\phi,\dot\phi)) + \sum_i \int_{D_i} \delta(\phi,\dot\phi,u(\phi),
F(\phi,\dot\phi)),
\end{equation}
and we will bound these terms separately.

{\bf The ``far'' region.}  For any $z \in C_{\far}$ we have
$r(z) \geq R$.  Assume $R$ is large enough so that \autoref{thm:u-exp}
and \autoref{thm:f-exp} apply.  Then we have
$u \approx \half \log \abs{P}$ and $F \approx \half \frac{\dot P}{P}$
with respective errors bounded by $a(\gamma) \e^{-\gamma R}$ and
$a(\gamma) \|\dot\phi\|_\sigma \e^{-\gamma R}$ for some constant
$a(\gamma)$.  If these approximate equalities were exact, then
$\delta$ would vanish identically; that is, by direct substitution
into the definition \eqref{eq:delta}, we find that
\begin{equation} \delta\left(\phi, \dot\phi, \half \log \abs{P}, \half
\frac{\dot{P}}{P}\right) = 0.\end{equation}
To handle the situation at hand, we will strengthen this to show that $\delta$ is pointwise small
when $u$ and $F$ are only \emph{near} $\half \log \abs{P}$ and $\half
\frac{\dot P}{P}$ (respectively).

Again by substitution into \eqref{eq:delta}, we find that for any
scalar functions $w$ and $\mu$ we have
\begin{equation}
\delta\left(\phi, \dot{\phi}, \half \log |P| + w, \half
\frac{\dot P}{P} + \mu\right)  = \left ( 2 \frac{|\dot{P}|^2}{|P|} (\e^{-2w}
-1) - 4 \frac{\e^{-2w}}{|P|} \re ( P \bar{\dot{P}} \mu )
\right ) \de x \wedge \de y.
\label{eq:delta-pointwise}
\end{equation}
Now assume that $|w| < 1$, so that
\begin{equation}
\label{eq:w-term}
2 |\e^{-2w} -1| \leq c |w|
\end{equation}
and
\begin{equation}
\label{eq:mu-term}
\left | 4 \e^{-2w} |P|^{-1} \re(P \bar{\dot{P}} \mu) \right | \leq c |\dot{P}| |\mu|
\end{equation}
for a constant $c > 0$.  Using these estimates with
\eqref{eq:delta-pointwise} gives
\begin{equation}
\label{eq:delta-linear-in-w-mu}
\left | \frac{\delta\left(\phi, \dot{\phi}, \half \log |P| + w, \half
  \frac{\dot{P}}{P} + \mu \right)}{\de x \wedge \de y} \right | \leq
c\left ( \frac{|\dot P|^2}{|P|} |w|
+ |\dot P| |\mu| \right ).
\end{equation}
If we furthermore assume $R > 1$, then
\autoref{lem:phi-pointwise} applies to $\phi$ throughout $C_{\far}$, giving
a uniform lower bound on $\e^{-2 \sigma} |P|$.  Substituting this into the
previous bound, we can now bound $\delta$ relative to the
hyperbolic area form as follows:
\begin{equation}
\left | \delta\left(\phi, \dot{\phi}, \half \log |P| + w, \half
\frac{\dot{P}}{P} + \mu \right) \right | \leq c
\left ( \abs{\dot{\phi}}_\sigma^2  |w| + \abs{\dot{\phi}}_\sigma |\mu| \right ) \e^{2 \sigma} |\de z|^2. \label{eq:delta-area-comp}
\end{equation}
Here $c$ is a constant, but not the same constant as in
\eqref{eq:delta-linear-in-w-mu}.  We already observed that on $C_{\far}$
the integrand $\delta(\phi,\dot{\phi},u(\phi),F(\phi,\dot{\phi}))$ has
the form \eqref{eq:delta-pointwise} with
$|w| \leq a(\gamma) \e^{-\gamma R}$ and
$|\mu| \leq a(\gamma) \|\dot\phi\|_\sigma \e^{-\gamma R}$.  Thus
\begin{equation}
\label{eq:delta-pointwise-actual}
\left |
\delta(\phi,\dot\phi,u(\phi),
F(\phi,\dot\phi)) \right | \leq c'(\gamma) |\phi|_\sigma^2 \e^{-\gamma
  R} \, \e^{2 \sigma} |\de z|^2 \text{ on }C_{\far},
\end{equation}
for a constant $c'(\gamma)$. Integrating \eqref{eq:delta-pointwise-actual},
and using that the $\sigma$-area of $C_{\far}$ is bounded and
$|\dot{\phi}|_\sigma \leq \|\dot{\phi}\|_\sigma$, we obtain
\begin{equation}
\label{eq:far-bound}	
\int_{C_\far} \delta(\phi,\dot\phi,u(\phi),
F(\phi,\dot\phi)) = O\left (\|\dot\phi\|_\sigma^2 \e^{-\gamma
  R} \right )
\end{equation}
with the implicit constant depending only on $c'(\gamma)$ from
\eqref{eq:delta-pointwise-actual}.  

{\bf The ``near'' region.}
Next we consider the integral over one of the disks $D_i$ in
\eqref{eqn:decomposition}.  Identify $D_i$ with a disk $\{ |z| < R \}$
in $\C$, using a coordinate $z$ in which $\left . \phi \right |_{D_i} = z
\, \de z^2$. 

On $D_i$ there is a unique holomorphic vector field $X = \chi \partial_z$ such that
$\dot \phi = \cL_X \phi = \cL_X (z \,\de z^2)$; explicitly, if we write
\begin{equation}
  \dot\phi = \sum_n a_n z^n \, \de z^2,
\end{equation}
then
\begin{equation}
\label{eq:explicit-X}
   \chi = \sum_n \frac{a_n}{2n+1} z^n.
\end{equation}
By \autoref{thm:holo-variation}, the associated function $F_X$ defined by
$F_X \e^{2 u(\phi)} |\de z|^2 = \cL_X(\e^{2 u(\phi)} |\de z|^2)$ satisfies
\eqref{eq:diff-F} on $D_i$, which is the same equation satisfied by
$F(\phi,\dot\phi)$.  We will show that $F_X$ and $F(\phi,\dot\phi)$
are in fact exponentially close on $D_i$.

First we consider the restrictions of these functions to
$\partial D_i$, which is far from the zeros of $\phi$, allowing the
estimates of the previous sections to be applied.  By
\autoref{thm:f-exp} we have
\begin{equation}
\label{eq:Fcompact-near-semiflat}
 \left | F(\phi,\dot\phi) - \frac{1}{2} \frac{\dot{P}}{P} \right | = O(
\|\dot{\phi}\|_\sigma \e^{-\gamma R} ) \text{ on } \partial D_i.
\end{equation}
Turning to $F_X = \chi_z + 2 \chi u_z$, note that the equation
$\dot\phi = \cL_X \phi$ gives 
\begin{equation}
 \chi_z = \frac{\dot{P} - \chi P_z}{2 P} = \frac{\dot{P}}{2
  P} - \chi \partial_z(\log |P|)
\end{equation}
and thus
\begin{equation}
F_X = \chi_z + 2 \chi u_z = \frac{\dot{P}}{2P} + 2
\chi \partial_z\left(u - \frac{1}{2} \log |P|\right).
\end{equation}
By the $C^1$ bound from \autoref{thm:u-exp} we have 
\begin{equation} \label{eq:ushift-estimate}
  \partial_z\left(u - \frac{1}{2} \log |P|\right) = O(\e^{-\gamma R}) \text{ on } \partial D_i.
\end{equation}
Next we need a bound on $\abs{\chi}$ on $\partial D_i$.
For this, note that for any $z \in D_i$,
$\chi(z)$ depends linearly on $\dot\phi$, and scales as
$t^{-\frac23}$. Thus, for $t > 1$ we have an estimate $\abs{\chi(z)} < c(z) \norm{\dot\phi}_\sigma$
for some $c(z)$, and since the closure of $D_i$ is compact we can take this constant
to be independent of $z$, i.e. on $D_i$ we have
\begin{equation} \label{eq:chiz-estimate}
	\abs{\chi} = O (\norm{\dot\phi}_\sigma).
\end{equation}
Now combining \eqref{eq:ushift-estimate} and \eqref{eq:chiz-estimate}, we get
\begin{equation}
\label{eq:Fplanar-near-semiflat}
 \left | F_X - \frac{1}{2} \frac{\dot{P}}{P} \right | = O(
\|\dot{\phi}\|_\sigma \e^{-\gamma R} ) \text{ on } \partial D_i.
\end{equation}
Then by \eqref{eq:Fcompact-near-semiflat} and \eqref{eq:Fplanar-near-semiflat}
we find that the function $\mu : D_i \to \C$ defined by
\begin{equation}\mu = F_X - F(\phi,\dot\phi)\end{equation}
satisfies
\begin{equation}
\label{eq:F-near-F-X-boundary}
\mu = O(\|\dot \phi\|_\sigma \e^{-\gamma R}) \text{
  on } \partial D_i.\end{equation} Because $F_X$ and
$F(\phi,\dot\phi)$ both satisfy the linear inhomogeneous equation
\eqref{eq:diff-F}, their difference $\mu$ satisfies the associated
homogeneous equation, which has the form $ (\Delta - k) \mu = 0$ for
an everywhere positive function $k$ (compare
\eqref{eq:diff-F-named-coefs}); this implies that $|\mu|$ has no
interior maximum.  Thus $|\mu|$ achieves its maximum on
$\partial D_i$, and \eqref{eq:F-near-F-X-boundary} gives
\begin{equation} \label{eq:mu-estimate}
 \mu = O(\|\dot \phi\|_\sigma \e^{-\gamma R}) \text{ on }D_i.
\end{equation}

Next we use this estimate on $\mu$ to estimate the integral of
$\delta(\phi, \dot\phi, u(\phi), F(\phi,\dot\phi))$ over $D_i$.
We have
\begin{equation}
\label{eq:delta-split}
\begin{split}
\delta(\phi, \dot\phi, u(\phi), F(\phi,\dot\phi)) &= 
\delta(\phi, \dot\phi, u(\phi), F_X + \mu)\\
&= \delta(\phi, \dot\phi, u(\phi), F_X) + 4 \e^{-2u(\phi)} \re(\mu P
\bar{\dot{P}}) \, \de x \wedge \de y.
\end{split}
\end{equation}
Since $\e^{-2 u(\phi)} |P| \leq 1$ (by
\autoref{lem:u-rough-lower}) we have
\begin{equation}
 |4 \e^{-2u(\phi)} \re(\mu P
\bar{\dot{P}}) | \leq 4 |\mu| |\dot{P}|
\end{equation}
and the bound \eqref{eq:mu-estimate} gives
\begin{equation}
 \int_{D_i} 4 \e^{-2u(\phi)} \re(\mu P
\bar{\dot{P}}) \, \de x \wedge \de y= O(\|\dot \phi\|^2_\sigma \e^{- \gamma R} ).
\end{equation}
Considering the other term on the right hand side of 
\eqref{eq:delta-split}, by \autoref{lem:exactness} the form
$\delta(\phi, \dot \phi, u, F_X)$ is exact, so we can use
Stokes's theorem to reduce to a boundary term:
\begin{equation} \label{eq:disc-int-ibp}
	 \int_{D_i} \delta(\phi, \dot \phi, u, F_X) = \int_{\partial D_i}
\left ( \e^{-2u} - |P|^{-1} \right )\left ( 2 |P|^2 \star \de
|\chi|^2 + |\chi|^2 \star \de |P|^2 \right).
\end{equation}
It just remains to show that this boundary term is exponentially
small. Fix some $\gamma'$ with $\gamma < \gamma' < 4$.
By \autoref{thm:u-exp}, we have $(\e^{-2u} -
  |P|^{-1}) = O(\e^{-\gamma' R})$ on $\partial D_i$. Next, using
  the estimate \eqref{eq:chiz-estimate}, the fact that $P$ scales as $t$,
  and the fact that the coordinate radius of $D_i$ scales as $t^{1/3}$, we have
\begin{equation}
 \int_{\partial D_i} \left ( 2 |P|^2 \star \de
|\chi|^2 + |\chi|^2 \star \de |P|^2 \right) = O(t^\frac73 \|\dot \phi\|^2_\sigma).
\end{equation}
Using this in \eqref{eq:disc-int-ibp} gives
\begin{equation}
 \int_{D_i} \delta(\phi, \dot \phi, u, F_X) = O(t^{\frac73} \|\dot \phi\|^2_\sigma \e^{-\gamma' R} ) = O(\|\dot \phi\|^2_\sigma \e^{-
  \gamma R} ),
\end{equation}
where in the last equality we use the fact that $R \to \infty$ as $t \to \infty$
by \eqref{eq:Rtdep}.

Now we have bounded the integrals of both terms in \eqref{eq:delta-split};
combining these bounds we conclude
\begin{equation}
\label{eq:near-bound}
	\int_{D_i} \delta(\phi, \dot\phi, u(\phi), F(\phi,\dot\phi)) =  O(\|\dot \phi\|^2_\sigma \,\e^{-
  \gamma R} ).
\end{equation}

{\bf Summing up.} 
Finally, substituting the far and near bounds (\eqref{eq:far-bound} and \eqref{eq:near-bound})
into \eqref{eqn:decomposition}, we obtain
\begin{equation}
\label{eqn:Delta-final}
\Delta(\phi, \dot{\phi}) = O(\|\dot{\phi}\|_\sigma ^2 \e^{-\gamma R}).
\end{equation}
Using \eqref{eq:Rtdep}, and that $\gamma < 4$ was arbitrary, we obtain
the exponential bound from \autoref{thm:main} by taking $\gamma =
8 \alpha / M(\phi_0)$.

Finally, we consider the contributions to the multiplicative constant
in \eqref{eqn:Delta-final}.  We have seen that the individual
exponential estimates in the components of $C_{\near}$ depend only on
$\gamma$.  The number of such components is linear in the genus of
$C$.  The estimate in $C_{\far}$ obtained above is also linear in the
hyperbolic area of $C$, or equivalently in the genus.  Overall we find
the multiplicative constant in the final estimate depends on
$\gamma$ and the topology of $C$, or equivalently, on $\alpha$,
$M(\phi_0)$, and the genus.  This completes the proof of \autoref{thm:main}.

\vfill
\pagebreak

\vspace{2em}

\noindent Department of Mathematics, Statistics, and Computer Science\\
University of Illinois at Chicago\\
Chicago, IL\\
\texttt{david@dumas.io}

\vspace{1em}

\noindent Department of Mathematics\\
University of Texas at Austin\\
Austin, TX\\
\texttt{neitzke@math.utexas.edu}

\end{document}